\documentclass{amsart}

\usepackage[latin9]{inputenc}
\usepackage{amsthm}
\usepackage{amssymb}

\newtheorem{thm}{Theorem}[section]
\newtheorem{lemma}[thm]{Lemma}

\theoremstyle{definition}
\newtheorem{definition}[thm]{Definition}

\theoremstyle{remark}
\newtheorem{rem}[thm]{Remark}

\numberwithin{equation}{section}

\newcommand{\N}{{\mathbb N}}
\newcommand{\R}{{\mathbb R}}

\newcommand{\Z}{{\mathbb Z}}

\newcommand{\cal}{\mathcal}

\newcommand{\cl}{c\ell}

\newcommand{\folner}{F\o{}lner}

\begin{document}

\title{Large sets in countable amenable groups and its application}

\author{Dibyendu De}

\address{Department of Mathematics, University of Kalyani, India}

\email{dibyendude@klyuniv.ac.in}

\thanks{The first author was partially supported DST PURSE Grant.}

\author{Pintu Debnath}
\address{Department of Mathematics, University of Kalyani, India}
\email{debnathpintu03@gmail.com}
\thanks{This work is a part of second authors Doctoral Dissertation and supported by CSIR grant}


\date{\today}


\begin{abstract}
In the present paper our main objective is to extend the notion of
$D$-sets in countable amenable groups and as a consequence we prove
that a D$^{*}$-set in $(\N,+)$ is a D-set in $(\N,\cdot)$. Further
we also discuss its connection with ergodicity and weak mixing for
amenable group actions. We present an example of an amenable group,
namely the group $(F_{q}[X],+)$, polynomial ring generated by finite
field $F_{q}$, where IP$^{*}$-sets, D$^{*}$-sets and Central$^{*}$-sets
are all equivalent.
\end{abstract}

\maketitle

\section{Introduction}

Central sets in $(\N,+)$ were introduced by Furstenberg \cite{refF}
and are known to have substantial combinatorial structure. For example,
any Central set contains arbitrarily long arithmetic progressions,
all possible finite sums of an infinite sequence, and solutions to
all partition regular systems of homogeneous linear equations. In
succession several notions of sets were introduce all of which have
rich combinatorial structure like Central sets, for example Quasi
Central sets \cite{refBuH}, C- sets \cite{refDHS}. Another inclusion
in this list is D-set \cite{refBD}. Furstenberg's original definition
of Central sets was in terms of notions of topological dynamics.
Later Bergelson and Hindman \cite{refBD} defined the notion of a
Central set in an arbitrary semigroup $(S,+)$ in terms of the algebra
of $\beta S$, the Stone-\v{C}echo compactification of $S$. They also
defined notion of dynamical Central set, using the natural extension
of Furstenberg's definition, and pointed out that any dynamical Central
sets in$(S,+)$ is Central in $(S,+)$. Moreover, a result of Weiss
(see \cite[Theorem 6.11]{refBH}) guarantees that in a countable
semigroup $(S,+)$, a subset of $S$ is Central if and only if it
is dynamically Central. This equivalence of dynamically central and
Central sets for arbitrary semigroup were finally proved in \cite{refSY}
.

On the contrary to Central sets the notion of Quasi Central and C-sets
sets were introduced using algebraic structure of $\beta S$. Equivalent
dynamical characterization were established in \cite{refBuH} and
\cite{refLi} respectively.

The notion of D-set was introduced by Bergelson and Downarowicz in
\cite{refBD} for subsets of $(\N,+)$ using algebraic structure
of $\beta\N$, and proved to be useful in connection with weak mixing
for $\Z$-action. Authors also described dynamical characterization
of D-sets in $(\N,+)$. Using dynamical characterization it can be
easily proved that central sets are quasi central and quasi central
sets are D-sets.

To describe the Central, Quasi Central and D-sets, we shall need to
pause and briefly introduce the algebraic structure of the Stone-\v{C}ech
compactification $S$ of a discrete semigroup $(S,+)$. Given any
discrete semigroup $S$ we take $\beta S$ to be the set of ultrafilters
on $S$, identifying the points of $S$ with the principal ultrafilters.
The topology on $\beta S$ has a basis consisting of $\{\cl{A}:A\subset S\}$,
where $\cl A=\{p\in\beta S:A\in p\}$. The operation on $S$ extends
to $\beta S$, making $S$ a left topological semigroup with $S$ contained
in its topological center. That is, for each $p\in\beta S$, the function
$\lambda_{p}(q)=p+q$ from $\beta S$ to $\beta S$ is continuous. And,
for each $x\in S$, the function $\rho_{x}:\beta S\to\beta S$
defined by $\rho_{x}(q)=q+x$ is continuous. (In spite of the fact
that we are denoting the extension by $+$, the operation on $S$
is very unlikely to be commutative. The center of $(\beta S,+)$ is equal to the center of
$S$. Given $p$ and $q$ in $\beta S$ and $A\subset S$, $A\in p+q$
if and only if $\{x\in S:-x+A\in p\}\in q$.

Any compact left topological semigroup $T$ has a smallest two sided
ideal $K(T)$ which is the union of all of the minimal left ideals
and is also the union of all of the minimal right ideals. The intersection
of any minimal left ideal with any minimal right ideal is a group
and therefore contains idempotents. In particular, there are idempotents
in $K(T)$. Such idempotents are called minimal. A subset $L$ of
$T$ is a minimal right ideal if and only if $L=p+T$ for some minimal
idempotent $p$. See \cite{refB}  and \cite{refB96}  for an elementary introduction
to the algebra of S and for any unfamiliar details. For any discrete
semigroup $S$. a set $A\subset S$ is an IP-set if and only if $A$
is a member of an idempotent in $\beta S$. $A$ is an IP$^{*}$ set
if and only if it is a member of every idempotent in $\beta S$. A
subset $C\subset S$ is called central if it belongs to some idempotents
of $K(\beta S)$. Like IP$^{*}$-set a subset $C\subset S$ is called
Central$^{*}$-set if it belongs to every minimal idempotent of $\beta S$.

The notion of $D$-sets first introduced by Bergelson and Downarowicz
in \cite{refBD} for $(\N,+)$. It was defined analogously to Central
-sets by replacing minimal idempotents by a wider class of idempotents
all of whose members have positive upper Banach density, so that the
class ${\cal D}$ of $D$- sets is (strictly) intermediate between
IP and Central sets. They also obtained a characterization of $D$-sets,
analogous to that of IP-sets and Central-sets. To introduce the notion
of D-sets let us recall the notion of upper Banach density.

\begin{definition}
A subset $A\subset Z$ is said to have positive upper Banach density
if
\[
\bar{d}(A)=\mbox{lim sup}_{(m-n)\to\infty}\frac{|A\cap[n,m-1]|}{m-n}>0.
\]

\end{definition}
An idempotent $p$ in $\beta\Z$ is called \textit{essential idempotent}
if every member of $p$ has positive upper Banach density. members
of \textit{essential idempotents} are called $D$-sets. Since the
notion of essential idempotents not only depends on the algebraic
structure of the Stone-\v{C}echo compactification, it's not straight
forward to generalize this notion for arbitrary semigroup like minimal
idempotents.

The paper is organized as follows. In section 2 we will introduce
the notion of essential idempotents for countable amenable group $G$
and then define the notion of $D$-sets in terms of topological dynamics.
Finally it will be proved that a subset of an amenable group $G$
is a $D$-set if and only if it belongs to some essential idempotents
of $\beta G$.

Like IP$^{*}$ and Central$^{*}$ set, a subset of $G$ will be called
$D^{*}$-set if it belongs to every essential idempotent of $\beta G$.
We denote the collection of all $D^{*}$-sets by ${\cal D}^{*}$ and
the union $\bigcup_{g\in G}(g{\cal D}^{*})$ by denote by ${\cal D}_{+}^{*}$.
In section 3 it will be shown that the familiar ergodic-theoretic
notions of ergodicity, weak mixing can be characterized by ${\cal D}^{*}$
and ${\cal D}_{+}^{*}$ sets.

Finally in section 4 we will discuss some results on Goldbach conjecture
on $\N$ and polynomial rings over finite fields. See \cite{refLe}
for an elementary introduction to polynomial rings over finite fields.

\section{Notion of $D$- sets for countable amenable group}

The notion of upper Banach density has a natural generalization for
 countable amenable groups. For this purpose let us first recall the
notion of discrete amenable group.
\begin{definition}

A discrete group $G$ is said to be amenable if there exists an invariant
mean on the space $B(G)$ of real-valued bounded functions on $G$,
that is, a positive linear functional $L:B(G)\to\R$ satisfying

1. $L(1_{G})=1,$

2. $L(f_{g})=L(_{g}f)=L(f)$ for all $f\in B(G)$ and $g\in G$, where
$f_{g}(t)=f(tg)$ and $_{g}f(t)=f(gt)$.
\end{definition}

The existence of an invariant mean is only one item from a long list
of equivalent properties. We will find the following characterization
of amenability  for discrete groups, which was established by \folner\ in \cite{refFo},
to be especially useful.

\begin{thm}
A countable group $G$ is amenable if and only if it has a left  \folner\
sequence, namely a sequence of finite sets $F_{n}\subset G$, $n\in\N$,
with $|F_{n}|\to\infty$ and such that

\[
\frac{|F_{n}\cap gF_{n}|}{|F_{n}|}\to1,\mbox{ for all }g\in G.
\]
\end{thm}
\begin{proof}
\cite[Corollary 5.3]{refNam}.
\end{proof}

\begin{definition}
Let $G$ be a countable amenable group. A subset $E$ of $G$ is said
to have positive \folner\ density or upper density with respect
to some  \folner\ sequence $\{F_{n}\}_{n\in\N}$ provided
that
\[
\mbox{lim sup}{}_{n\to\infty}\frac{|E\cap F_{n}|}{|F_{n}|}>0.
\]

This will be denoted by $\overline{d}_{F_{n}}(E)$.
\end{definition}

\begin{definition}
Let $G$ be a countable amenable group. A subset $E$ of $G$ is said
to have positive upper density if
$$\bar{d}(E)=\text{sup}\{d_{F_{n}}(E): F_{n}\text{ is a \folner\ sequence in }G\}>0.$$

\end{definition}

From \cite[Theorem 3]{refFo} one can show that a set $A\subseteq G$ has positive upper density
with respect to some \folner\ sequence if and only if there
exists an invariant mean $L$ on $B(G)$ such that $L(1_{A})>0$.

\begin{definition}

By a dynamical system we mean a pair $(X,\,\langle T_{g}\rangle_{g\in G})$ where
\begin{enumerate}
\item $G$ is a  semigroup
\item $\langle T_{g}\rangle_{g\in G}$ is continuous for all $g\in G$.
\item  $g,h\in G$, we have $T_g\circ T_h = T_{gh}$ .
\end{enumerate}
\end{definition}
\begin{definition}
A point $y$ contained in the dynamical system $(X,\,\langle T_{g}\rangle_{g\in G})$
is said to be essentially recurrent if the set of visits $\{g\in G:T_{g}y\in U_{y}\}$
for any neighborhood $U_{y}$ of $y$ has positive upper density.
\end{definition}

This can be easily observe that in any $G$ be a countable amenable
group syndetic sets always have positive upper density. This shows
that every uniformly recurrent point is essentially recurrent. A characterization
of essentially recurrent points in terms of the properties of their
orbit closures is provided below.

\begin{definition}
A dynamical system $(Y,\,\langle T_{g}\rangle_{g\in G})$ will be
called measure saturated if every nonempty open set $U$ there exists
an invariant measure $\mu$ such that $\mu(U)>0$.
\end{definition}

\begin{thm}
A point $y$ in a dynamical system $(Y,\,\langle T_{g}\rangle_{g\in G})$
is essentially recurrent if and only if the orbit closure $\overline{O}(y)=\overline{\{T_{g}(y):g\in G\}}$
is measure saturated.
\end{thm}

\begin{proof}
First let us show that if a point $y$ in a dynamical system $(Y,\,\langle T_{g}\rangle_{g\in G})$
is essentially recurrent then the orbit closure $\overline{O}(y)=\overline{\{T_{g}(y):g\in G\}}$
is measure saturated. Let $U_{y}$ be an open neighborhood of $y$
containing a closed neighborhood $U$. Since $y$ is essentially recurrent,
the set $A=\{g\in G:T_{g}y\in U_{y}\}$ has positive upper density
$d$ with respect to a \folner\ sequence $\{F_{n}\}_{n\in\N}$ in $G$. Then 
\[
\mbox{lim sup }{}_{n\to\infty}\frac{|E\cap F_{n}|}{|F_{n}|}>0.
\]
Passing to a subsequence we can say that there exists a sequence $\{F_{n}\}_{n\in\N}$
in $G$ with $|F_{n}|\to\infty$, such that $\frac{|E\cap F_{n}|}{|F_{n}|}$
converges to $d$. Let $\{\mu_{n}:n\in\N\}$ be the normalized counting
measures supported by the sets $\{T_{g}(y):g\in F_{n}\}$. Then it
is easy to observe that for each $n\in\N$ we have $\mu_{n}\circ T_{g}=\mu_{n}$.
Let $\mu$ be a weak limit of the sequence $(\mu_{n})_{n\in\N}$,
where a sequence of measures $(\mu_{n})_{n\in\N}$ converges to $\mu$
weak if $\int f\mu_{n}\to\int f\mu$ for every continuous function
$f$ on the space $Y$. Clearly $\mu$ is also $T_{g}$ invariant
for each $g\in G$ and supported by $\overline{O}(y)=\{T_{g}(y):y\in G\}$
and satisfies $\mu(U)>0$, and thus $\mu(U_{y})>0$.

Now for any invariant measure $\mu$ carried by $\overline{O}(y)$
we define

\[
M_{\mu}=\{x\in\overline{O}(y):\mbox{there exists a neighborhood }V\mbox{ of }x\mbox{ such that }\mu(V)>0\}.
\]
Let $M$ be the closure of union of all $M_{\mu}$'s. Then $y\in M$
and since $M$ is a closed invariant set, it follows that $M=\overline{O}(y)$,
i.e., $\overline{O}(y)$ is measure saturated.

Conversely, assume that $\overline{O}(y)$ is measure saturated. Let
$U_{y}\ni y$ be an open set. Then there exist an invariant measure
$\mu$ supported by $\overline{O}(y)$ such that $\mu(U_{y})=a>0$.
We have to show that the set $\{g\in G:T_{g}y\in U_{y}\}$ has positive
upper density. Since $G$ is a countable amenable group there exists
a  \folner\ sequence say $(F_{n})_{n\in\N}$, namely a sequence
of finite sets $F_{n}\subset G$, $n\in\N$ with $|F_{n}|\to\infty$
and such that $\frac{|F_{n}\cap gF_{n}|}{|F_{n}|}\to1$ for all $g\in G$.
Let us now set
\[
f_{n}(x)=\frac{1}{|F_{n}|}\sum_{g\in F_{n}}1_{U_{y}}(T_{g}(x)).
\]

Then we have that $0\leq f_{n}(x)\leq1$ for all $x$ and since $T_{g}$'s
are measure preserving we have that $\int f_{n}d\mu\geq a>0$ for
all $n\in\N$. Let $f(x)=\mbox{lim sup }{}_{n\to\infty}f_{n}(x)$.
By Fatso's Lemma, we have
\[
\int fd\mu=\int\mbox{lim sup }{}_{n\to\infty}f_{n}(x)d\mu\geq\mbox{lim sup }{}_{n\to\infty}\int f_{n}(x)\geq a.
\]

Since $\mu$ is supported on $\overline{O}(y)$ there exists $y^{\prime}\in\overline{O}(y)$
such that $f(y^{\prime})=a>0$. Now let us set $R=\{g\in G:T_{g}y^{\prime}\in U_{y}\}$.
Since

\[
f(y^{\prime})=\mbox{lim sup}_{n\to\infty}\frac{1}{|F_{n}|}\sum_{g\in F_{n}}1_{U_{y}}(T_{g}(x)),
\]

it follows that $R$ has \folner\ density $a$. Now for any $g\in G$
with $T_{g}y^{\prime}\in U_{y}$ there exists a neighborhood $V$
of $y^{\prime}$ such that $T_{g}(V)\subset U_{y}$. Choose $g^{\prime}\in G$
such that $T_{g^{\prime}}y\in V$ so that $T_{gg^{\prime}}y\in U_{y}$.
It follows that $\{g\in G:T_{g}y\in U_{y}\}$ has positive upper density
and hence $y$ is essentially recurrent point.
\end{proof}

Let us  extend the notion of essential idempotent in $\beta\Z$
for countable amenable group.

\begin{definition}
Let $G$ be a countable amenable group. An idempotent $p\in\beta G$
is said to be an essential idempotent if every member of $p$ has
positive upper density.
\end{definition}

For any $x\in G$, let us set $r_{x}:G\to G$ by $r_{x}(g)=gx$.
Now $\beta G$ becomes a compact left topological semigroup under the natural extension $\lambda_{q}(p)=q\cdot p$.
Let us first  characterize essentially recurrent points in the dynamical
system $(\beta G,\langle\lambda_{q}\rangle_{q\in \beta G})$.

\begin{lemma}
Let $G$ be a countable amenable group. An idempotent $q$ in $\beta G$
is an essentially recurrent point in the topological dynamical system
\textup{$(\beta G,\langle\lambda_{q}\rangle_{q\in \beta G})$,} if
and only it is essential idempotent in $(\beta G,\cdot)$.
\end{lemma}

\begin{proof}
Let $q$ be essentially recurrent point in the topological dynamical
system $(\beta G,\langle\lambda_{q}\rangle_{q\in \beta G})$ and let
$E$ be any element of $q$. We have to show that $E$ has positive
upper density in the amenable group $G$. The closure $\overline{E}$
of $E$ in $\beta G$ can be interpreted as a neighborhood of $q$.
Since $q$ is essentially recurrent, $\overline{O}(q)$ is measure
saturated and hence exists an invariant measure $\mu$ such that $\mu(\overline{E})>0$.
Choose $g\in G$ such that $g\in\overline{E}$, $g\in E$. Since $\mu$
is supported by the orbit closure of the identity the set $\{g\in G:\lambda_{g}(e)\in\overline{E}\}$
has positive upper density $(\beta G,\langle\lambda_{p}\rangle_{p\in \beta G})$.
But the $\{g\in G:\lambda_{g}(e)\in\overline{E}\}=E$. This implies
that $E$ has positive upper density, and hence $q$ is an essential
idempotent in $(\beta G,\cdot)$.

To prove the converse consider the map defined by $\lambda_{q}(p)=q\cdot p$
onto $\overline{O}(q)$ and both $e$ and $q$ map to $q$. A neighborhood
$U_{q}$ of $q$ in $\overline{O}(q)$ lifts to a neighborhood $V_{q}$
of $q$ in $\beta G$ and the set $R_{q}=\{g\in G:\lambda_{g}(q)\in U_{q}\}$
contains the set $R_{e}=\{g\in G:\lambda_{g}(e)\in V_{q}\}$
in fact $g\in V_{q}$ implies that $\lambda_{g}(q)\in U_{q}$
as $qU_{q}=V_{q}$. But the set $R_{e}$ is a member of $q$ (because
its complement is not). Since $q$ is assumed to be an essential idempotent,
all members of $q$ have positive upper density. It follows that $R_{e}$
has positive upper density and hence, $q$ is essentially recurrent.
\end{proof}
To prove our main theorem let us recall the following lemma from \cite{refBH}.
\begin{lemma}
Let $\pi:X\to Y$ be a topological factor map (surjection) between
dynamical systems $(X,S)$ and $(Y,T)$. If $y$ is an essentially
recurrent point in $Y$ then there exists an essentially recurrent
$\pi$-lift $x$ of $y$. Moreover, we can find such $x$ for which
$\overline{O}(x)$ contains no proper closed invariant subset which
is mapped by $\pi$ onto $\overline{O}(y)$.
\end{lemma}
Now we are in a position to prove our main Theorem of this section.
\begin{thm}
Let $G$ be a countable amenable group. A set $D\subset G$ is a $D$-set
if and only there exists a compact dynamical system $(X,\langle T_{g}\rangle_{g\in G})$,
points $x,y\in X$ with $x,y$ proximal and $y$ essentially recurrent
and an open neighborhood $U_{y}$ of $y$ such that

\[
D=\{g\in G:T_{g}(x)\in U_{y}\}.
\]
\end{thm}
\begin{proof}
Let $D=\{g\in G:T_{g}(x)\in U_{y}\}$ where $x$, $y$ and $U_{y,}$
are as in the formulation of the theorem. Since $x,y$ are proximal
$p\mbox{-lim}_{g\in G}T_{g}x=p\mbox{-lim}_{g\in G}T_{g}y$. Consider
a factor map $\pi:\beta G\to\overline{O}(y)$ defined by $p\to p\mbox{-lim}_{g\in G}T_{g}y$.
By previous lemma we can find in $\beta G$ an essentially recurrent
point $p_{1}$ which is $\pi$-lift of $y$, and whose orbit closure
is a minimal lift of $\overline{O}(y)$. We will show that $p_{1}$
can be replaced by an idempotent. Consider the set

\[
I=\{p\in\overline{O}(p_{1}):\pi(p)=y\}.
\]

It can be easily verified that $I$ is a closed sub semigroup of $\beta G$,
so it contains an idempotent $q$. Since $\pi(p)=y$, its orbit closure
maps onto $\overline{O}(y)$. By minimality of the lift $\overline{O}(p_{1})$,
$q$ has the same orbit closure as $p_{1}$, and hence is essentially
recurrent.

Since $\pi(q)=q\mbox{-lim}_{g\in G}T_{g}y=p\mbox{-lim}_{g\in G}T_{g}x$
and $U_{y}$ is a neighborhood of $y$ we have that $\{g\in G:T_{g}x\in U_{y}\}\in q$
and therefore $D\in q$.

Conversely we consider the dynamical system $(X,\langle T_{g}\rangle_{g\in G})$,
where $X=\{0,1\}^{G}$ and for each $g\in G$, $T_{g}:X\to X$ defined
by $T_{g}(f)=f\circ\lambda_{g}$. Let $D$ be a $\mbox{D}$-set. Then
there exists an essential idempotent $q\in G$ such that $D\in q$.
Consider the characteristic function $x=1_{D}\in X$. Let $y=q\mbox{-lim }T_{g}x$.
Since $q$ is an idempotent $y=q\mbox{-lim }T_{g}y=q\mbox{-lim }T_{g}(q\mbox{-lim }T_{g}(x))$.
Since $y$ is the image of $q$ via the factor map $\pi:\beta G\to\overline{O}(y)$
given by $p\to p\mbox{-lim}T_{g}y$ and $q$ is essentially recurrent
point in $(\beta G,\langle\hat{\sigma}_{x}\rangle_{x\in G})$, we
have that $q$ is essentially recurrent point.

Now let $U_{y}=\{z\in X:z(e)=y(e)\}$ is a neighborhood of $y$ in
$X$. Then $U_{y}$ is a neighborhood of $y$ in $X$. We note that
$y(e)=1$. In fact $y=q\mbox{-lim }T_{g}(x)$ so that $\{g\in G:T_{g}x\in U\}\in q$.
This implies that $\{g\in G:T_{g}x\in U_{y}\}\cap D\neq\emptyset$.
Let us choose $g\in D$ such that $T_{g}(x)\in U_{y}$. Then $y(e)=T_{g}(x)(e)=x(eg)=1$.
Thus given any $g\in G$,

\[
\begin{array}{ccc}
g\in D & \Leftrightarrow & x(g)=1\\
 & \Leftrightarrow & T_{g}(x)(e)=1\\
 & \Leftrightarrow & T_{g}(x)\in U_{y}.
\end{array}
\]

\end{proof}
This theorem implies that every Central set in an amenable group is
a D-set. We also have the following combined additive and multiplicative
property.
\begin{thm}
Every $D^{*}$-set in $(\N,+)$ is a $D$-set in $(\N,\cdot)$.
\end{thm}

\section{Mixing Properties}

In this section we will discuss now the connections between essential
idempotents and unitary actions of countable amenable groups. We shall
consider a unitary representation $(U_{g})_{g\in G}$ on the Hilbert
space $L^{2}(X,{\cal B},\mu)$ defined by $U_{g}f(x)=f(T_{g}x)$.
\begin{definition}
Let $(U_{g})_{g\in G}$ be a unitary representation of a group $G$
on a separable Hilbert space $H$.

(i) A vector $x\in H$ is called compact if the set $\{U_{g}x:g\in G\}$
is totally bounded in $H$.

(ii) The representation $(U_{g})_{g\in G}$ is called weak mixing
if there are no nonzero compact vectors.
\end{definition}

\begin{thm}
Given a unitary representation $(U_{g})_{g\in G}$ of a group $G$
on a separable Hilbert space $H$ let
\[
H_{c}=\{x\in H:f\,\, is\,\, compact\,\, with\,\, respect\,\, to\,\,(U_{g})_{g\in G}\}.
\]
Then the restriction $(U_{g})_{g\in G}$ to the invariant subspace
$H_{wm}=H_{c}^{\bot}$ is weak mixing.
\end{thm}
Recall that in a Hilbert space the norm convergence $\mbox{lim}x_{n}=y$
is equivalent to the conjunction of the weak convergence of $x_{n}$
to $y$ and the convergence of norms $\mbox{lim}\Vert x_{n}\Vert=\Vert y\Vert$.
Since any unitary operator $U$ is an isometry, the relation $p$-$\mbox{lim}U_{g}x=x$
for some $p\in\beta G$ holds in the weak topology if and only if
it holds in the strong topology. The following lemma follows from
\cite[Theorem 4.3]{refB}.
\begin{lemma}
\label{id con1}If $p\in\beta G$ is an idempotent then for any $x\in{\cal H}_{c}$
one has $p$-\textup{$\mbox{lim}U_{g}x=x$}.
\end{lemma}
The above statement can be reversed for essential idempotents.
\begin{lemma}
\label{id con 2}If $p\in\beta G$ is an essential idempotent and
$p$-\textup{$\mbox{lim}U_{g}x=x$} for some $x\in{\cal H}$ then
$x\in{\cal H}_{c}$.\end{lemma}
\begin{proof}
For $\epsilon>0$ consider the set $E=$$\{g\in G:\Vert U_{g}x-x\Vert<\epsilon/2\}$.
Since $p$-$\mbox{lim}U_{g}x=x$, we have $E\in p.$ Note that for
any $g_{1},g_{2}\in E$ one has

\[
\Vert T_{g_{1}g_{2}^{-1}}x-x\Vert=\Vert T_{g_{1}}x-T_{g_{2}}x\Vert\leq\Vert T_{g_{1}}x-x\Vert+\Vert T_{g_{2}}x-x\Vert\leq\epsilon.
\]
Since $E$ has positive upper Banach density, this implies that $EE^{-1}$
is syndetic. Since $U$ is an isometry we have covered the orbit of
$x$ by finitely many $\epsilon-$balls, hence the orbit of $x$ is
precontract, so that $x\in H_{c}$.\end{proof}
\begin{lemma}
If $p\in\beta G$ is an essential idempotent then for any $x\in H_{wm}$
one has $p$-\textup{$\mbox{lim}U_{g}x=0$} weak.
\end{lemma}

\begin{proof}
By compactness of the ball of radios $\Vert x\Vert$ around zero
in the weak topology,there exists some y such that $p$-$\mbox{lim}U_{g}x=y$
weak. Since $H_{wm}$ is closed and invariant under $U$, we have
$y\in H_{wm}.$ On the other hand, $p$ is an idempotent, by Lemma
\ref{id con1} $p$-$\mbox{lim}U_{g}y=y$. By Lemma \ref{id con 2},
$y\in H_{c}$. This implies $y=0$.
\end{proof}
One can show that unitary representation $(U_{g})_{g\in G}$ is weak
mixing if and only if in the decomposition $H=H_{c}\oplus H_{wm}$
one has $H_{c}=\{0\}$. Let now $(X,{\cal B},\mu,\langle T_{g}\rangle_{g\in G})$
be an invertible weakly mixing system. It can be easily checked that
in this case the unitary operator induced by $\langle T_{g}\rangle_{g\in G}$
on $L^{2}(\mu)$ is weakly mixing in the above sense on the orthocomplement
of the space of constant functions. This allows us to work with the
unitary operator $(U_{g})_{g\in G}$ acting on the Hilbert space $H$
as $L_{0}^{2}(X,{\cal B},\mu)=\{f\in L^{2}(X,{\cal B},\mu):\int fd\mu=0\}$.
\begin{rem}
It follows from the previous discussions that an invertible probability
measure preserving system $(X,{\cal B},\mu,\langle T_{g}\rangle_{g\in G})$
is weakly mixing if and only if for any $A,B\in{\cal B}$ and any
essentially recurrent idempotent $p$, $p$-$\mbox{lim}_{g}\mu(A\cap T_{g}B)=\mu\left(A\right)\mu\left(B\right)$.
\end{rem}
Let us recall here deep Theorem of Lindenstrauss for our purpose.
\begin{definition}
A sequence of sets $F_{n}$ will be said to be tempered if for some
$C>0$ and all $n$, $|{\displaystyle \bigcup_{k\leq n}F_{k}^{-1}F_{n+1}|}\leq C|F_{n+1}|$.
\end{definition}
\begin{thm}[Lindenstrauss Pointwise ergodic theorem]
Let G be an amenable group acting ergodically on a measure space
$(X,{\cal B},\mu)$, and let $F_{n}$ be a tempered \folner\ sequence. Then for any $f\in L^{1}(\mu)$,

\[
\mbox{lim}_{n}\frac{1}{|F_{n}|}\sum_{g\in F_{n}}f(gx)=\int f\left(x\right)d\mu\left(x\right)\mbox{ a. e.}
\]

\end{thm}

\begin{thm}
An invertible measure preserving system $(X,{\cal B},\mu,\langle T_{g}\rangle_{g\in G})$
is ergodic if and only if for any $A,B\in{\cal B}$ and $\epsilon>0$
the set $R_{A,B}^{\epsilon}=\{g\in G:\mu(A\cap T_{g}B)>\mu\left(A\right)\mu\left(B\right)-\epsilon\}$
belongs to ${\cal D}_{+}^{*}$.
\end{thm}
To prove the above Theorem we need the following Lemma. From now to express  ``measure preserving system'' we will also write m.p.s.
\begin{lemma}
Any invertible probability m.p.s. $(X,{\cal B},\mu,\langle T_{g}\rangle_{g\in G})$
is ergodic if and only if $N(A,B)=\{g\in G:\mu(T_{g}A\cap B)>0\}\neq\emptyset$
for any $A,B\in{\cal B}$ such that $\mu(A)\mu(B)>0$.\end{lemma}
\begin{proof}
Let $(X,{\cal B},\mu,\langle T_{g}\rangle_{g\in G})$ be a ergodic
measure preserving system. Then for any tempered \folner\ sequence
 $\langle F_{n}\rangle_{n\in\N}$,

\[
\mbox{lim}_{n}\frac{1}{|F_{n}|}\sum_{g\in F_{n}}1_{A}(T_{g}x)=\int1_{A}\left(x\right)d\mu\left(x\right)=\mu(A)\mbox{ a.e. }
\]

Then dominated convergence theorem we have

\[
\mbox{lim}_{n}\frac{1}{|F_{n}|}\sum_{g\in F_{n}}\int1_{A}(T_{g}x)1_{B}\left(x\right)d\mu\left(x\right)=\mu(A)\int1_{B}d\mu=\mu\left(A\right)\mu\left(B\right)>0.
\]

i.e.

\[
\mbox{lim}_{n}\frac{1}{|F_{n}|}\sum_{g\in F_{n}}\mu(T_{g}A\cap B)=\mu\left(A\right)\mu\left(B\right)>0.
\]

Hence there exists $g_{0}\in G$ such that $\mu(T_{g_{0}}A\cap B)>0$.
This implies, $N(A,B)\neq\emptyset$.

Conversely let $(X,{\cal B},\mu,\langle T_{g}\rangle_{g\in G})$ is
not ergodic. Then there exists $A\in{\cal B}$ with $0<\mu(A)<1$
such that $T_{g}A=A$ for all $g\in A$. This implies that $0=\mu(A\cap A^{c})=\mu(T_{g}A\cap A^{c})$.
This contradicts the fact that $N(A,B)\neq\emptyset$.
\end{proof}

\begin{proof}[Proof of Theorem 3.9.]
Assume that $(X,{\cal B},\mu,\langle T_{g}\rangle_{g\in G})$ is
ergodic. Denote $f=1_{A}$ and $h=1_{B}$. Decompose $h=h_{1}+h_{2},$
$h_{1}\in H_{c}$, $h_{2}\in H_{wm}.$ Clearly $\intop h_{1}d\mu=\mu\left(B\right)$.
Now let $\langle F_{n}\rangle{}_{n\in\N}$ be tempered \folner\ sequence.
Then

\[
\frac{1}{|F_{n}|}\sum_{g\in F_{n}}f(T_{g}x)\rightarrow\int f\left(x\right)d\mu\left(x\right)=\mu(A)\mbox{ a.e. }.
\]
from dominated convergence theorem we have

\[
\frac{1}{|F_{n}|}\sum_{g\in F_{n}}\int f(T_{g}x)g_{1}\left(x\right)d\mu\left(x\right)\rightarrow\mu(A)\int g_{1}d\mu=\mu\left(A\right)\mu\left(B\right).
\]
hence there exist $g_{o}$ satisfying $\intop f(T_{g_{0}}x)g_{1}\left(x\right)d\mu\left(x\right)>\mu\left(A\right)\mu\left(B\right)-\epsilon$.
Let $p$ be an essential idempotent. Applying Lemma \ref{id con1}
and \ref{id con 2}, we can write

\[
p\mbox{-lim}\mu(T_{g_{0}}A\cap T_{g}B)=p\mbox{-lim}\int f(T_{g_{0}}x)g(T_{g}x)d\mu\left(x\right)
\]

\[
=p\mbox{-lim}\int f(T_{g_{0}}x)g_{1}(T_{g}x)d\mu\left(x\right)+p\mbox{-lim}\int f(T_{g_{0}}x)g_{2}(T_{g}x)d\mu\left(x\right)
\]

\[
=\int f(T_{g_{0}}x)g_{1}\left(x\right)>\mu\left(A\right)\mu\left(B\right)-\epsilon.
\]
this implies that $g_{0}^{-1}R_{A,B}^{\epsilon}\in p,$which proves
that $R_{A,,B}^{\epsilon}$ is $D_{+}^{*}$ set.

The converse implication is obvious : if the sets $R_{A,B}^{\epsilon}$
are $D_{+}^{*}$ then they are nonempty which implies ergodicity.
Previous Lemma.\end{proof}
\begin{thm}
The system $(X,{\cal B},\mu,\langle T_{g}\rangle_{g\in G})$ is weakly
mixing if and only if for any $A,B\in{\cal B}$ and $\epsilon>0$
the set $R_{A,B}^{\epsilon}$ is $D^{*}$. \end{thm}
\begin{proof}
Assume that $(X,{\cal B},\mu,\langle T_{g}\rangle_{g\in G})$ is weakly
mixing. Then by Remark 3.5 for any $A,B\in{\cal B}$ and any essential
idempotent $p$, we have $p$-$\mbox{lim}_{g}\mu(A\cap T_{g}B)=\mu\left(A\right)\mu\left(B\right)$
and hence $R_{A,B}^{\epsilon}$ is a $D^{*}$ set.

To prove the converse part let us assume that

\[
L_{A,B}^{\epsilon}=\{g\in G:\mu(A\cap T_{g}B)<\mu\left(A\right)\mu\left(B\right)+\epsilon\}.
\]
It proved that $L_{A,B}^{\epsilon}$ is also a $D^{*}$-set. Now
$g\in R_{A^{c},B}^{\epsilon}$ implies that $\mu(A^{c}\cap T_{g}B)>\mu\left(A^{c}\right)\mu\left(B\right)-\epsilon$.
But

\[
\mu(A^{c}\cap T_{g}B)+\mu(A\cap T_{g}B)=\mu(T_{g}B)=\mu(B).
\]
So

\[
g\in R_{A^{c},B}^{\epsilon}\Rightarrow\mu\left(B\right)-\mu(A\cap T_{g}B)>\mu\left(A^{c}\right)\mu\left(B\right)-\epsilon\Rightarrow\mu(A\cap T_{g}B)<\mu\left(A\right)\mu\left(B\right)+\epsilon.
\]
Hence $L_{A,B}^{\epsilon}$ is $D^{*}$-set .

Now $L_{A,B}^{\epsilon}\cap R_{A,B}^{\epsilon}$$=$$\{g\in G:|\mu(A\cap T_{g}B)-\mu\left(A\right)\mu\left(B\right)|<\epsilon\}$
is $D^{*}$-set. Thus for any essential idempotent $p$, we have
$p$-$\mbox{lim}_{g}\mu(A\cap T_{g}B)=\mu\left(A\right)\mu\left(B\right)$.
This implies $(X,{\cal B},\mu,\langle T_{g}\rangle_{g\in G})$ is
weak mixing.\end{proof}
\begin{rem}
It is worth to note that if the system $(X,{\cal B},\mu,\langle T_{g}\rangle_{g\in G})$
is weak mixing then $R_{A,B}^{\epsilon}$ is of density one. In fact
from \cite{refBR} we know that for countable discrete amenable group
$G$ the system $(X,{\cal B},\mu,\langle T_{g}\rangle_{g\in G})$
is weak mixing iff for any \folner\ sequence $\{F_{n}\}$ and $A,B\in B$
we have $\frac{1}{\mid F_{n}\mid}{\displaystyle \sum_{g\in F_{n}}}|\mu(A\cap T_{g}B)-\mu\left(A\right)\mu\left(B\right)|\rightarrow0$
as $n\rightarrow\infty$. From this it can be estimated that $R_{A,B}^{\epsilon}\cap L_{A,B}^{\epsilon}$
is of density one i.e $R_{A,B}^{\epsilon}$ is of density one. This
can also be proved by using Lindenstrauss pointwise ergodic Theorem.\end{rem}
\begin{thm}
Let $G$ countable discrete amenable group acting on a probability
space $(X,B,\mu)$. Then $(X,{\cal B},\mu,\langle T_{g}\rangle_{g\in G})$
is weak mixing iff for any \folner \ sequence $\langle F_{n}\rangle$
and $A,B\in{\cal B}$, $R_{A,B}^{\epsilon}$ is of density one.\end{thm}
\begin{proof}
Since $T_{g}$ is weak mixing in particular ergodic for any tempered
\folner \ sequence $\langle F_{n}\rangle$ in $G$

\[
\frac{1}{|F_{n}|}\sum_{g\in F_{n}}1_{B}(T_{g}x)\rightarrow\int1_{B}\left(x\right)d\mu\left(x\right)=\mu(B)\mbox{ a.e. }
\]

Now applying dominated convergence theorem, we have

\[
\frac{1}{|F_{n}|}\sum_{g\in F_{n}}\int1_{B}(T_{g}x)1_{A}\left(x\right)d\mu\left(x\right)\rightarrow\mu(B)\int1_{A}d\mu=\mu\left(A\right)\mu\left(B\right).
\]

Hence

\[
\frac{1}{\mid F_{n}\mid}\sum_{g\in F_{n}}\mu(A\cap T_{g}B)\to\mu\left(A\right)\mu\left(B\right).
\]
Now using the equality $(A\cap T_{g}B)\times(A\cap T_{g}B)=\left(A\times A\right)\cap(T_{g}\times T_{g})\left(B\times B\right)$
we have

\[
\frac{1}{\mid F_{n}\mid}\sum_{g\in F_{n}}(\mu\times\mu)((A\cap T_{g}B)\times(A\cap T_{g}B))
\]

\[
=\frac{1}{\mid F_{n}\mid}\sum_{g\in F_{n}}\mu\times\mu(\left(A\times A\right)\cap T_{g}\times T_{g}\left(B\times B\right))\rightarrow\mu\times\mu\left(A\times A\right)\mu\times\mu\left(B\times B\right)
\]

hence

\[
\frac{1}{\mid F_{n}\mid}\sum_{g\in F_{n}}\mu(A\cap T_{g}B)^{2}\rightarrow\mu\left(A\right)^{2}\mu\left(B\right)^{2}
\]

Thus

\[
\frac{1}{\mid F_{n}\mid}\sum_{g\in F_{n}}(\mu(A\cap T_{g}B)-\mu\left(A\right)\mu\left(B\right))^{2}
\]

\[
=\frac{1}{\mid F_{n}\mid}\sum_{g\in F_{n}}\mu(A\cap T_{g}B)^{2}-2\mu\left(A\right)\mu\left(B\right)\frac{1}{\mid F_{n}\mid}\sum_{g\in F_{n}}\mu(A\cap T_{g}B)+\mu\left(A\right)^{2}\mu\left(B\right)^{2}\rightarrow0.
\]

Now if possible let $\frac{|S\cap F_{n}|}{|F_{n}|}\rightarrow l>0$,
where $S=G\setminus(R_{A,B}^{\epsilon}\cap L_{A,B}^{\epsilon})$.
Then

\[
\frac{1}{\mid F_{n}\mid}\sum_{g\in F_{n}}(\mu(A\cap T_{g}B)-\mu\left(A\right)\mu\left(B\right))^{2}
\]

\[
\geq\frac{1}{\mid F_{n}\mid}\sum_{g\in S\cap F_{n}}(\mu\left(A\cap T_{g}B\right)-\mu\left(A\right)\mu\left(B\right))^{2}
\]

\[
\geq\frac{1}{|F_{n}|}\sum_{g\in S\cap F_{n}}\epsilon^{2}=\epsilon^{2}\frac{|S\cap F_{n}|}{|F_{n}|}\rightarrow\epsilon^{2}l>0,
\]

a contradiction .
\end{proof}

\section{Few Words on Goldbach Conjecture}

If Goldbach conjecture is true then we have $2\N\subset P+P$. Author
in \cite{refEs} established that $2\N\setminus P+P$ is of natural
density zero. This shows that if $A$ be a Central-set in $\N$ it intersects
$2\N$ centrally i.e.  in a set with positive upper density and therefore meets
$P+P$. Hence $P+P$ becomes a Central$^{*}$-set. But it is unknown to
us, whether $P+P$ is an IP$^{*}$-set.

In the following we provide an example of D$^{*}$-sets in polynomial
ring over finite fields $F_{q}$ where $\mbox{Char }F_{q}\neq2$.
In case of $(\Z,+)$ we know that any ideal generated by an integer
is an IP$^{*}$-set. However in the following Theorem we shall extend
this result for $(F_{q}\left[X_{1},X_{2},\ldots,X_{k}\right],+)$
in its full generality.

\begin{thm}
\label{IPstar}The ideal $\langle f_{1}(X_{1}),\ldots,f_{k}(X_{k})\rangle$
generated by $f_{1}(X_{1})$, $\ldots$,
$f_{k}(X_{k})$ in the polynomial ring $(F_{q}\left[X_{1},X_{2},\ldots,X_{k}\right],+,\cdot)$
over finite field\textup{ $F_{q}$} is an IP$^{*}$-set. in the group
$(F_{q}\left[X_{1},X_{2},\ldots,X_{k}\right],+)$. \end{thm}

\begin{proof}
We will show that ideal $\langle f_{1}(X_{1}),f_{2}(X_{2}),\ldots,f_{k}(X_{k})\rangle$
is an IP$^{*}$-set and hence a D$^{*}$-set in $(F_{q}\left[X_{1},X_{2},\ldots,X_{k}\right],+)$.

For simplicity we work with $k=2$. Let $\langle g_{n}(X_{1},X_{2})\rangle_{n=1}^{\infty}$be
a sequence in $F_{q}[X_{1},X_{2}]$.

Let $g(X_{1},X_{2})$ be a polynomial in $F_{q}[X_{1},X_{2}]$. Then

\[
g(X_{1},X_{2})=\sum_{i\leq n,j\leq m}a_{i,j}X_{1}^{i}X_{2}^{j},\mbox{ where }a_{i,j}\in F_{q}.
\]

Since $X_{1}^{i},f_{1}(X_{1})\in F_{q}[X_{1}]$ and $X_{2}^{j},f_{2}(X_{2})\in F_{q}[X_{2}]$
by applying division algorithm we have
\[
X_{1}^{i}=f_{1}(X_{1})q_{i}(X_{1})+r_{i}(X_{1})\mbox{, where deg}(r_{i}(X_{1}))<\mbox{deg}f_{1}(X_{1})
\]
\[
X_{2}^{j}=f_{2}(X_{2})q_{j}(X_{2})+r_{j}(X_{2})\mbox{, where deg}(r_{j}(X_{2}))<\mbox{deg}f_{2}(X_{2}).
\]
Then $g(X_{1},X_{2})$ can be expressed as
\[
\sum_{i\leq n,j\leq m}a_{i,j}(f_{1}(X_{1})q_{i}(X_{1})+r_{i}(X_{1}))(f_{2}(X_{2})q_{j}(X_{2})+r_{j}(X_{2})).
\]
\[
g(X_{1},X_{2})=f_{1}(X_{1})h_{1}(X_{1},X_{2})+f_{2}(X_{2})h_{2}(X_{1},X_{2})
\]
\[
+\sum_{\begin{array}{c}
\mbox{ deg}(r_{i}(X_{1}))<\mbox{deg}f_{1}(X_{1})\\
\mbox{deg}(r_{j}(X_{2}))<\mbox{deg}f_{2}(X_{2})
\end{array}}a_{i,j}r_{i}(X_{1})r_{j}(X_{2}).
\]
Therefore we can write

\[
g(X_{1},X_{2})=h(X_{1},X_{2})+r(X_{1},X_{2}),
\]
where

\[
h(X_{1},X_{2})\in\langle f_{1}(X_{1}),f_{2}(X_{2})\rangle
\]
and $r(X_{1},X_{2})$ is a polynomial such that $\mbox{deg}r(X_{1},X_{2})<\mbox{deg}f_{1}(X_{1})+\mbox{deg}f_{2}(X_{2})$.

This implies that

\[
g_{n}(X_{1},X_{2})=h_{n}(X_{1},X_{2})+r_{n}(X_{1},X_{2})
\]
where

\[
h_{n}(X_{1},X_{2})\in\langle f_{1}(X_{1}),f_{2}(X_{2})\rangle
\]
and $r_{n}(X_{1},X_{2})$ is a polynomial such that $\mbox{deg}r_{n}(X_{1},X_{2})<\mbox{deg}f_{1}(X_{1})+\mbox{deg}f_{2}(X_{2})$.

But the set $\{r_{n}(X_{1},X_{2}):n\in\N\}$ is finite. Since $\{g_{n}(X_{1},X_{2}):n\in\N\}$
is infinite there exists $p$ many polynomials $g_{n_{i}}(X_{1},X_{2}):i=1,2,\ldots,p$
such that the corresponding $r_{n_{i}}(X_{1},X_{2})$ for $i=1,2,\ldots,p$
are equals. Now adding we get
\[
\sum_{i=1}^{p}g_{n_{i}}(X_{1},X_{2})=\sum_{i=1}^{p}h_{n_{i}}(X_{1},X_{2})+\sum_{i=1}^{p}r_{n_{i}}(X_{1},X_{2}).
\]
This implies that
$${\displaystyle \sum_{i=1}^{p}g_{n_{i}}(X_{1},X_{2})}\in\langle f_{1}(X_{1}),f_{2}(X_{2})\rangle$$
as
$${\displaystyle \sum_{i=1}^{p}r_{n_{i}}(X_{1},X_{2})=0}.$$
 Therefore
$\langle f_{1}(X_{1}),f_{2}(X_{2})\rangle$ is an IP$^{*}$-set and
hence D$^{*}$-set.
\end{proof}
Let us state first the analogue Goldbach conjecture for the function
field $F_{q}[X]$.\\

\noindent \textbf{Goldbach conjecture for the polynomials over finite field
:} For a monic polynomial $m(X)$ with degree $\geq2$, there exist
two monic irreducible polynomial $f_{1}(X)$ and $f_{2}(X)$ with
$\mbox{deg }f_{2}(X)<\mbox{deg }f_{1}(X)=\mbox{deg }m(X)$ such that
$m(X)=f_{1}(X)+f_{2}(X)$.\\

For any polynomial $m(X)$ there exists $\alpha\in F_{q}$ such that $\alpha m(X)$
is monic and if above conjecture is true then $\alpha m(X)$=$f_{1}(X)$ $+f_{2}(X)$
$\Rightarrow m(X)=\alpha^{-1}f_{1}(X)+\alpha^{-1}f_{2}(X)$. Hence
this implies that

\[
P_{X}+P_{X}=\{f_{1}(X)+f_{2}(X):\mbox{ irreducible polynomial }f_{1}(X)\mbox{ and }f_{2}(X)\}
\]
is a Central$^{*}$-set. But it is not known to us that whether like $P+P$,
the following set $P_{X}+P_{X}$ is a Central$^{*}$-set. We conclude this
article wit the following thm which shows that $^{*}$notions are
equivalent.

\begin{thm}
\label{syndetic IP}Any $IP^{*}$-set containing $0$ in $(F_{q}[X],+)$
contains an ideal of the form $\langle X^{m}\rangle$ for some $m\in\N$.
\end{thm}

\begin{proof}
Let us claim that any syndetic IP set $A$ in $(F_{q}[X],+)$ contains
$\langle X^{m}\rangle$ for some $m\in\N$. Now $A$ being a syndetic
set will be of the form

\[
A={\displaystyle \bigcup_{i=1}^{k}(f_{i}(X)+\langle X^{m}\rangle)}\mbox{ (modulo finite terms) }
\]
for some $m,k\in\N$ with $m>\mbox{deg}f_{i}(X)$. Again since $A$
is an IP-set one of $f_{i}(X)$ must be zero.

We end this article with the following observation.\end{proof}
\begin{rem}
We know that for arbitrary countable discrete amenable group, IP$^{*}$$\Rightarrow$
D$^{*}$$\Rightarrow$C$^{*}$. In $(\N,+)$ the converse implications
do not hold. But from Theorem \ref{IPstar} and \ref{syndetic IP}
it follows that converse implications do hold for $F_{q}[X]$. \end{rem}

\bibliographystyle{amsplain}

\end{document}